\documentclass[12pt,reqno]{amsart}

\usepackage{hyperref,setspace}
\usepackage{fullpage}

\numberwithin{equation}{section}

\newtheorem{thm}{Theorem}
\newtheorem{lem}[thm]{Lemma}

\newtheorem{cor}[thm]{Corollary}

\newcommand{\R}{\mathbb R}
\newcommand{\C}{\mathbb C}
\newcommand{\Z}{\mathbb Z}
\newcommand{\eps}{\varepsilon}

\def\deg{\operatorname{deg}}
\def\Vol{\operatorname{Vol}}

\begin{document}

\title{Vitushkin-type theorems}

\author{Omer Friedland}
\address{Institut de Math\'ematiques de Jussieu, Universit\'e Pierre et Marie Curie (Paris 6), 4 Place Jussieu, 75005 Paris, France.}
\email{friedland@math.jussieu.fr}
\thanks{}

\author{Yosef Yomdin}
\address{Department of Mathematics, The Weizmann Institute of Science, Rehovot 76100, Israel.}
\email{yosef.yomdin@weizmann.ac.il}
\thanks{This research was partially supported by ISF grant No. 639/09 and by the Minerva foundation.}

\begin{abstract}
It is shown that for a subset $A\subset\R^n$ that has the global Gabrielov property, a Vitushkin-type estimate holds. Concrete examples are given for sub-level sets of certain classes of functions.
\end{abstract}

\subjclass[2010]{14P10, 26B15 (primary), and 14R10, 26D05, 42105 (secondary)}
\keywords{Metric entropy, Gabrielov property, Vitushkin-type estimates, Sub-level sets}

\maketitle

\section{Introduction}

The metric entropy of a subset $A \subset \R^n$ can be bounded in terms of the $i$-dimensional ``size'' of $A$. Indeed, the theory of multi-dimensional variations, developed by Vitushkin \cite{V1,V2}, Ivanov \cite{I}, and other, provides a bound by measuring the $i$-dimensional ``size'' of $A$ in terms of its variations.

\smallskip

Let us recall a general definition of the metric entropy of a set. Let $X$ be a metric space, $A \subset X$ a relatively compact subset. For every $\eps> 0$, denote by $M(\eps, A)$ the minimal number of closed balls of radius $\eps$ in $X$, covering $A$ (note that this number does exist because $A$ is relatively compact). The real number $H_{\eps}(A) = \log M(\eps,A)$ is called the $\eps$-entropy of the set $A$. In our setting, we assume $X=\R^n$, and it will be convenient to modify slightly this definition, and consider coverings by the $\eps$-cubes $Q_\eps$, which are translations of the standard $\eps$-cube, $Q_\eps^n = [0,\eps]^n$, that is, the $\frac \eps 2$-ball in the $\ell_\infty$ norm.

\smallskip

The following inequality, which we refer to as Vitushkin's bound, bounds the metric entropy of a set $A$ in terms of its multi-dimensional variations, that is, for every $A\subset \R^n$ it holds
\begin{align} \label{eq:V}
M(\eps,A) \le c(n) \sum_{i=0}^n V_i(A) / {\eps^i},
\end{align}
where the $i$-th variation of $A$, $V_i(A)$, is the average of the number of connected components of the section $A\cap P$ over all $(n-i)$-affine planes $P$ in $\R^n$. In particular, our definition of the metric entropy implies that the last term in (\ref{eq:V}) has the form ${\mu_n(A)}/{\eps^n}$, where $\mu_n(A)$ denotes the $n$-dimensional Lebesgue measure (or the volume) of the set $A$.

\smallskip

Vitushkin's bound is sometimes considered as a difficult result, mainly because of the so-called multi-dimensional variations which are used. However, in some cases (cf. \cite{YC, Y,FY}) the proof is indeed very short and transparent. In this note, we present a Vitushkin-type theorem, which works in situations where we can control the number of connected components of the sections $A\cap P$ over all $(n-i)$-affine planes $P$. In section 2 we present our main observation that we can replace the $i$-th variation $V_i(A)$ of $A$, with an upper bound on the number of connected components of the section $A\cap P$ over all $(n-i)$-affine planes $P$. As in some cases, it is easier to compute this upper bound rather than $V_i$ which is the average. In Section \ref{sec:ent}, we extend Vitushkin's bound for semi-algebraic sets, to sub-level sets of functions, for which we have a certain replacement of the polynomial B\'ezout theorem. These results can be proved using a general result of Vitushkin in \cite{V1,V2} through the use of multi-dimensional variations. However, in our specific case the results below are much simpler and shorter and produce explicit (``in one step") constants.

\section{Gabrielov property and Vitushkin's bound} \label{sec:V}

In this section we establish a relation between Vitushkin's bound and the Gabrielov property of a set $A$. We show that an a priori knowledge about the maximal number of connected components of a set $A$, intersected with every $\ell$-affine plane in $\R^n$, allows us to estimate the metric entropy of $A$.

\smallskip

More precisely, we say that a subset $A\subset \R^n$ has the local Gabrielov property if for $a \in A$ there exist a neighborhood $U$ of $a$ and an integer $\hat C_\ell$ such that for every $\ell$-affine plane $P$, the number of connected components of $U\cap A\cap P$ is bounded by $\hat C_\ell$. If we can take $U=\R^n$, we say that $A$ has the global Gabrielov property. For example, every tame set has the local Gabrielov property (for more details see \cite{YC}).

\smallskip

The following theorem is applicable to arbitrary subsets $A\subset Q_1^n$. The boundary $\partial A$ of $A$ is defined as the intersection of the closures of $A$ and of $Q_1^n\setminus A$.

\begin{thm} [Vitushkin-type theorem] \label{thm:V}
Let $A \subset Q_1^n$ and let $0<\eps\le 1$. Assume that the boundary $\partial A$ of $A$ has the global Gabrielov property, with explicit bound $\hat C_\ell$ for $0\le \ell \le n$. Then
$$
M(\eps, A) \le C_0 + C_1/\eps + \cdots + C_{n-1} /\eps^{n-1} + \mu_n(A)/\eps^n,
$$
where $C_t:= \hat C_{n-t} 2^{t} \binom{n}{t}$.
\end{thm}

\begin{proof}
Let us subdivide $Q_1^n$ into adjacent $\eps$-cubes $Q_{\eps}^n$, with respect to the standard Cartesian coordinate system. Each $Q_{\eps}^n$, having a non-empty intersection with $A$, is either entirely contained in $A$, or it intersects the boundary $\partial A$ of $A$. Certainly, the number of those cubes $Q_{\eps}^n$, which are entirely contained in $A$, is bounded by $\mu_n(A) / \mu_n(Q_{\eps}) = \mu_n(A) / \eps^n$. In the other case, in which $Q_{\eps}^n$ intersects $\partial A$, it means that there exist faces of $Q_{\eps}^n$ that have a non-empty intersection with $\partial A$. Among all these faces, let us take one with the smallest dimension $s$, and denote it by $F$. In other words, there exists an $s$-face $F$ of the smallest dimension $s$ that intersects $\partial A$, for some $s=0,1,\ldots,n$. Let us fix an $s$-affine plane $V$, which corresponds to $F$. Then, by the minimality of $s$, $F$ contains completely some of the connected components of $\partial A\cap V$, otherwise $\partial A$ would intersect a face of $Q_{\eps}^n$ of a dimension strictly less than $s$. By our assumption, the number of connected components with respect to an $s$-affine plane is bounded by $\hat C_s$. According to the subdivision of $Q_1^n$ to $Q_\eps$ cubes, we have at most $\left( \frac {1}{\eps} + 1 \right)^{n-s}$ $s$-affine planes with respect to the same $s$ coordinates, and the number of different choices of $s$ coordinates is $\binom{n}{s}$. It means that the number of cubes, that have an $s$-face $F$ which contains completely some connected component of $A\cap V$, is at most
$$
\hat C_s \binom{n}{s} \left( \frac {1}{\eps} + 1 \right)^{n-s} \le \hat C_s 2^{n-s} \binom{n}{s}/\eps^{n-s} .
$$

Let us define the constant
$$
C_{n-s}:= \hat C_s 2^{n-s} \binom{n}{s} .
$$

Note that $C_0$ is the bound on the number of cubes that contain completely some of the connected components of $A$. Thus, we have
$$
M(\eps, A) \le C_0 + C_1/\eps + \cdots + C_{n-1} /\eps^{n-1} + \mu_n(A)/\eps^n .
$$

This completes the proof.
\end{proof}

\section{Entropy estimates of sub-level sets} \label{sec:ent}

In this section we extend Vitushkin's bound to sub-level sets of certain natural classes of functions, beyond polynomials. We do so by ``counting'' the singularities of these functions, and then bounding the number of the connected components of their sub-level set through the number of singularities.

\smallskip

We start with a general simple ``meta''-lemma, which implies, together with a specific computation of the bound on the number of singularities, all our specific results below. Consider a class of functions $\mathcal F$ on $\R^n$. We assume that $\mathcal F$ is closed with respect to taking partial derivatives, restrictions to affine subspaces of $\R^n$, and with respect to sufficiently rich perturbations. There are many classes of functions that comply with this condition, for example, we may speak about the class of real polynomials of $n$ variables and degree $d$, and the classes considered below in this section. Assume that for each $f_1,,\ldots,f_n \in {\mathcal F}$ the number of non-degenerate solutions of the system
\begin{align*}
f_1 = f_2 = \cdots = f_n = 0,
\end{align*}
is bounded by the constant $C(D(f_1,\ldots,f_n))$, where $D(f_1,\ldots,f_n)$ is a collection of ``combinatorial'' data of $f_i$, like degrees, which we call a ``Diagram'' of $f_1,\ldots,f_n$. We assume that the diagram is stable with respect to the deformations we use. In each of the examples below we define the appropriate diagram specifically.

\smallskip

Let $f \in {\mathcal F}$. Denote by
\begin{align} \label{eq:sublevel}
W = W(f,\rho) = \{ x \in Q_1^n:f(x) \le \rho \},
\end{align}
the $\rho$-sub-level set of $f$, and let $\hat C_s=C(D({{\partial f(x)}\over {\partial x_1}},\ldots,{{\partial f(x)}\over {\partial x_s}})), \ s=1,\dots,n.$

\begin{lem} \label{lem:main}
The boundary $\partial W$ has the global Gabrielov property, i.e. the number of connected components of $\partial W \cap P$, where $P$ is an $s$-affine plane in $\R^n$, is bounded by $\hat C_s$.
\end{lem}

\begin{proof}
We may assume that $P$ is a parallel translation of the coordinate plane in $\R^n$ generated by $x_{j_1},\dots,x_{j_s}$. Inside each connected component of $\partial W \cap P$ there is a critical point of $f$ restricted to $P$ (its local maximum or minimum), which is defined by the system of equations
$$
{{\partial f(x)}\over {\partial x_{j_1}}} = \cdots = {{\partial f(x)}\over {\partial x_{j_s}}} = 0 .
$$

After a small perturbation of $f$, we can always assume that all such critical points are non-degenerate. Hence by the assumptions above, the number of these points, and therefore the number of connected components, is bounded by
$\hat C_s.$ \end{proof}

\smallskip

In other words, $W$ has the Gabrielov property, with explicit bound $\hat C_\ell$. Therefore, Theorem \ref{thm:V} can be applied to this set, and under the assumptions above, we have

\begin{cor} \label{cor:main}
Let $0<\eps \le 1$. Then
$$
M(\eps, W) \le C_0 + C_1/\eps + \cdots + C_{n-1} /\eps^{n-1} + \mu_n(W)/\eps^n,
$$
where $C_t:= \hat C_{n-t} 2^{t} \binom{n}{t}$.
\end{cor}

\section{Concrete bounds on $\hat C_s$}

In view of Corollary \ref{cor:main}, our main goal now is to give concrete bounds on constants $\hat C_s$ in specific situations.

\subsection{Polynomials and B\'ezout's theorem}

Let $p(x) = p(x_1,\ldots, x_n)$ be a polynomial in $\R^n$ of degree $d$. We consider the sub-level set $W(p,\rho)$ as defined in (\ref{eq:sublevel}). Clearly, in this situation, by B\'ezout's theorem, we have
$$
\hat C_s(W(p,\rho)) \le (d-s)^s.
$$

\subsection{Laurent polynomials and Newton polytypes}

Let $\alpha\in\Z^n$. A Laurent monomial in the variables $x_1,\ldots,x_n$ is $x^\alpha = x_1^{\alpha_1}\cdots x_n^{\alpha_n}$. A Laurent polynomial is a finite sum of Laurent monomials,
$$
p(x) = p(x_1,\ldots,x_n) = \sum_{\alpha \in A \subset \Z^n} a_\alpha x^\alpha .
$$

The Newton polytope of $p$ is the polytope
$$
N(p) = \text{conv} \{ \alpha \in \Z^n ~\big|~ a_\alpha \neq 0\} .
$$

A natural generalization of the B\'ezout bound above is the following Bernstein-Ku\v shnirenko bound for polynomial systems with the prescribed Newton polytope.

\begin{thm} [\cite{Ku,B}] \label{Newt.diag}
Let $f_1,\ldots,f_n$ be Laurent polynomials with the Newton polytope $N\subset \R^n$. Then the number of non-degenerate solutions of the system
$$
f_1 = f_2 = \cdots = f_n = 0,
$$
is at most $n!\Vol_n(N)$.
\end{thm}

The Newton polytope of a general polynomial of degree $d$ is the simplex
$$
\Delta_d=\{\alpha \in \R^n, |\alpha|\leq d \} .
$$

Its volume is ${d^n}/{n!}$, and the bound of Theorem \ref{Newt.diag} coincides with the B\'ezout's bound $d^n$. The notion of the Newton polytope is connected to a representation of the polynomial $p$ in a fixed coordinate system $x_1,\ldots, x_n$ in $\R^n$. As we perform a coordinate changes (which may be necessary when we restrict a polynomial $p$ to a certain affine subspace $P$ of $\R^n$), $N(p)$ may change strongly. However, in Theorem \ref{thm:V} we restrict our functions only to affine subspaces $P$ spanned by a part of the standard basis vectors in a fixed coordinate system. The following lemma describes the behavior of the Newton polytope of $p$ under such restrictions, and under partial differentiation.

\begin{lem} \label{Newt.diag.restr}
The Newton polytope $N({{\partial p}\over{\partial x_i}})$ is $N_i(p)$ obtained by a translation of $N(p)$ to the vector $-e_i$, where $e_1,\ldots,e_n$ are the vectors of the standard basis in $\R^d$. The Newton polytope $N(p|P)$ of a restriction of $p$ to $P$, where $P$ is a translation of a certain coordinate subspace, is contained in the projection $\pi_P(N(p))$ of $N(p)$ on $P$.
\end{lem}

\begin{proof}
The proof of the first claim is immediate. In a restriction of $p$ to $P$, we substitute some of the $x_i$'s for their specific values. The degrees of the free variables remain the same.
\end{proof}

\smallskip

For a Newton polytope $N\subset \R^n$ define
$$
C_s(N):= \max \{ \Vol_s(N_{P_s}) ~ \big| \text{ $s$-dimensional coordinate subspaces $P_s$} \},
$$
where $N_{P_s}$ is the convex hull of the sets $\pi_{P_s}(N_i(p))$, for all the coordinate directions in $P_s$. Here, $\pi_{P_s}(N_i(p))$ is the projection of $N$ to $P_s$, shifted by $-1$ in one of coordinate directions $x_i$ in $P_s$.

\begin{thm} \label{Newt.diag.Thm}
Let $p$ be a Laurent polynomial with the Newton polytope $N$. Then for $s=1,\ldots,n$
$$
\hat C_s(W(p,\rho)) \le \frac{C_s(N)}{s!}.
$$
\end{thm}

\begin{proof}
According to the {\it proof} of Theorem \ref{thm:V}, the constants $\hat C_s(W(p,\rho))$ do not exceed the number of solutions in $Q^n_1$ of the system ${{\partial p(x)}\over {\partial x_{j_1}}} = \cdots = {{\partial p(x)}\over {\partial x_{j_s}}} = 0.$ Now application of Theorem \ref{Newt.diag}, Lemma \ref{Newt.diag.restr}, and of the definition of $C_s(N)$ above, completes the proof.
\end{proof}

\smallskip

An important example is provided by ``multi-degree $d$'' polynomials. A polynomial $p(x) = p(x_1,\ldots, x_n)$ is called multi-degree $d$, if each of its variable enters $p$ with degrees at most $d$. The total degree of such $p$ may be $nd$. In particular, multi-linear polynomials contain each variable with the degree at most one. Multi-linear polynomials appear in various problems in Mathematics and Computer Science, in particular, since the determinant of a matrix is a multi-linear polynomial in its entries.

\begin{thm} \label{multid}
Let $p$ be a multi-degree $d$ polynomial. Then for $s=1,\ldots,n$
$$
\hat C_s(W(p,\rho)) \le \frac{d^s}{s!}.
$$
\end{thm}

\begin{proof}
The Newton polytope $N(p)$ of a multi-degree $d$ polynomial $p$ is a cube $Q^n_d$ with the edge $d$. Its projection to each $P_s$ is $Q^s_d$. After the shift by $-1$ in one of the coordinate directions in $P_s$, it remains in $Q^s_d$. So $Q^s_d$ can be taken as the common Newton polytope of ${{\partial p(x)}\over {\partial x_{j_1}}}, \cdots, {{\partial p(x)}\over {\partial x_{j_s}}}.$ Application of Theorem \ref{Newt.diag.Thm} completes the proof.
\end{proof}

\subsection{Quasi-polynomials and Khovanskii's theorem}

Let $f_1, \ldots, f_k \in (\C^n)^*$ be a pairwise different set of complex linear functionals $f_j$ which we identify with the scalar products $f_j\cdot z, \ z=(z_1,\ldots,z_n)\in \C^n$. We shall write $f_j=a_j+i b_j$. A quasi-polynomial is a finite sum
$$
p(z) = \sum_{j=1}^k p_j(z) e^{f_j\cdot z},
$$
where $p_j \in \C[z_1,\ldots,z_n]$ are polynomials in $z$ of degrees $d_j$. The degree of $p$ is $m = \deg p = \sum_{j=1}^k (d_j + 1)$.

\smallskip

Below we consider $p(x)$ for the real variables $x=(x_1,\ldots,x_n)\in \R^n$, and we are interested in the following sub-level set of $p$ which is defined as $\{ x \in Q_1^n:|p(x)| \le \rho \}$. Denote $q(x)=|p(x)|^2$, then this sub-level set is also defined by $W(q,\rho^2)$.

\smallskip

A simple observation that $q(x)=|p(x)|^2=p(x)\bar p(x)$ tells us that we can rewrite $q$ as follows

\begin{lem}
$q(x)$ is a real exponential trigonometric quasi-polynomial with $P_{ij},Q_{ij}$ real polynomials in $x$ of degree $d_i + d_j$, and at most $\kappa:= k(k+1)/2$ exponents, sinus and cosinus elements. Moreover,
\begin{align*}
q(x) = \sum_{0\le i \le j \le k} e^{\langle a_i+a_j, x\rangle } \big[ P_{ij}(x) \sin \langle b_{ij}, x \rangle + Q_{ij}(x) \cos \langle b_{ij}, x \rangle \big],
\end{align*}
where $b_{ij} = b_i-b_j \in \R^n$.
\end{lem}

Now, we need to bound the singularities of $q$. This can be done using the following theorem due to Khovanskii, which gives an estimate of the number of solutions of a system of real exponential trigonometric quasi-polynomials.

\begin{thm}[Khovanskii bound \cite{K}, Section 1.4] \label{thm:khovanskii}
Let $P_1 = \cdots = P_n = 0$ be a system of $n$ equations with $n$ real unknowns $x = x_1,\ldots,x_n$, where $P_i$ is polynomial of degree $m_i$ in $n+k+2p$ real variables $x$, $y_1,\ldots,y_k$, $u_1,\ldots,u_p$, $v_1, \ldots, v_p$, where $y_i = \exp\langle a_j,x\rangle$, $j=1,\ldots,k$ and $u_q = \sin\langle b_q,x\rangle$, $v_q = \cos\langle b_q,x\rangle$, $q=1,\ldots,p$. Then the number of non-degenerate solutions of this system in the region bounded by the inequalities $|\langle b_q, x\rangle| < \pi/2$, $q=1,\ldots,p$, is finite and less than
$$
m_1 \cdots m_n \left( \sum m_i + p + 1\right)^{p+k} 2^{p+(p+k)(p+k-1)/2} .
$$
\end{thm}

\smallskip

Clearly, all the partial derivatives ${{\partial q(x)}\over {\partial x_j}}$ have exactly the same form as $q$. Therefore, Khovanskii's theorem gives the following bound on the number of critical points of $q$. More precisely, we have

\begin{lem}\label{khovanskii}
Let $V$ be a parallel translation of the coordinate subspace in $\R^n$ generated by $x_{j_1},\dots,x_{j_s}$. Then the number of non-degenerate real solutions in $V \cap Q_\rho^n$ of the system
$$
{{\partial q(x)}\over {\partial x_{j_1}}}= \cdots = {{\partial q(x)}\over {\partial x_{j_s}}}= 0,
$$
is at most
$$
({2\over \pi}\sqrt s \rho \lambda)^s \prod_{r=1}^s (d_{j_r} + d_{i_r})\left( \sum_{r=1}^s d_{j_r} + d_{i_r} + 2\kappa + 1\right)^{2\kappa} 2^{\kappa+(2\kappa)(2\kappa-1)/2},
$$
where $\lambda:= \max \Vert b_{ij} \Vert$ is the maximal frequency in $q$.
\end{lem}

\begin{proof}
The following geometric construction is required by the Khovanskii bound:Let $Q_{ij} = \{ x \in \R^n, |\langle b_{ij}, x\rangle| \leq {\pi\over 2}\}$ and let $Q= \bigcap_{0\le i \le j \le k} Q_{ij}$. For every $B \subset \R^n$ we define $M(B)$ as the minimal number of translations of $Q$ covering $B$. For an affine subspace $V$ of $\R^n$ we define $M(B\cap V)$ as the minimal number of translations of $Q\cap V$ covering $B\cap V$. Notice that for $B=Q^n_r$, a cube of size $r$, we have $M(Q^n_r)\leq ({2\over \pi}\sqrt n r \lambda)^n$. Indeed, $Q$ always contains a ball of radius ${\pi \over {2\lambda}}$. Now, applying the Khovanskii's theorem on the system
$$
{{\partial q(x)}\over {\partial x_{j_1}}}= \cdots = {{\partial q(x)}\over {\partial x_{j_s}}} = 0,
$$
we get that the number of non-degenerate real solutions in $V \cap Q_\rho^n$ is at most
\begin{align*}
({2\over \pi}\sqrt s \rho \lambda)^s \prod_{r=1}^s (d_{j_r} + d_{i_r})\left( \sum_{r=1}^s d_{j_r} + d_{i_r} + 2\kappa + 1\right)^{2\kappa} 2^{\kappa+(2\kappa)(2\kappa-1)/2}.
\end{align*}

Note that this bound is given in term of the ``diagram'' of $q$, and therefore of $p$.
\end{proof}

\begin{thm}
Let $p(x)$ a real quasi-polynomial (as described above). Then for $s=1,\ldots,n$
\begin{align*}
\hat C_s & (W(q,\rho^2)) \\
& \le ({2\over \pi}\sqrt s \rho \lambda)^s \prod_{r=1}^s (d_{j_r} + d_{i_r})\left( \sum_{r=1}^s d_{j_r} + d_{i_r} + 2\kappa + 1\right)^{2\kappa} 2^{\kappa+(2\kappa)(2\kappa-1)/2},
\end{align*}
where $q(x)=|p(x)|^2$.
\end{thm}

\subsection{Exponential polynomials and Nazarov's lemma}

In a particular case where $p$ is an exponential polynomial, that is,
$$
p(t) = \sum_{k=0}^m c_k e^{\lambda_k t},
$$
where $c_k,\lambda_k \in \C$, $t\in \R$. We can avoid Khovanskii's theorem and instead use the following result of Nazarov \cite[Lemma 4.2]{N}, which gives a bound on the local distribution of zeroes of an exponential polynomial.

\begin{lem}\label{lem:naz}
The number of zeroes of $p(z)$ inside each disk of radius $r>0$ does not exceed
$$
4m+7\hat \lambda r,
$$
where $\hat \lambda = \max|\lambda_k|$.
\end{lem}

Note that this result is applicable only in dimension $1$. Let $B\subset \R$ be an interval. Therefore, the number of real solutions of $p(t) \le \rho$ inside the interval $B$ does not exceed Nazarov's bound, that is, $4m+7\hat \lambda \mu_1(B)$, which gives us
$$
\hat C_1(W(p,\rho)) \le 4m+7\hat\lambda.
$$

\smallskip

For the case of a real exponential polynomial $p(t) = \sum_{k=0}^m c_k e^{\lambda_k t}$, $c_k,\lambda_k \in \R$, we get an especially simple and sharp result, as the number of zeroes of a real exponential polynomial is always bounded by its degree $m$ (indeed, the ``monomials" $e^{\lambda_k t}$ form a Chebyshev system on each real interval).

\subsection{Semialgebraic and tame sets}
We conclude with a remark about even more general settings for which Theorem \ref{thm:V} is applicable.

\smallskip

A set $A\subset\R^n$ is called semialgebraic, if it is defined by a finite sequence of polynomial equations and inequalities, or any finite union of such sets. More precisely, $A$ can be represented in a form $A=\bigcup_{i=1}^k A_i$ with $A_i=\bigcap_{j=1}^{j_i} A_{ij}$, where each $A_{ij}$ has the form
$$
\{x\in\R^n:p_{ij}(x)>0\} \text{ or } \{x\in\R^n:p_{ij}(x)\ge0\},
$$
where $p_{ij}$ is a polynomial of degree $d_{ij}$. The diagram $D(A)$ of $A$ is the collective data
$$
D(A) = (n,k,j_1,\ldots,j_k,(d_{ij})_{i=1,\ldots,k,~j=1,\ldots,j_i}).
$$

A classical result tells us that the number of connected components of a plane section $A\cap P$ is uniformly bounded. More precisely, we have
\begin{thm}[\cite{YC}] \label{thm:semi}
Let $A\subset\R^n$ be a semialgebraic set with diagram $D(A)$. Then the number of connected components of $A\cap P$, where $P$ is an $\ell$-affine plane of $\R^n$, is bounded by
$$
\hat C_\ell \le \frac12\sum_{i=1}^k(d_i+2)(d_i+1)^{\ell-1},
$$
where $d_i=\sum_{j=1}^{j_i} d_{ij}$.
\end{thm}

In other words, Theorem \ref{thm:semi} says that any semialgebraic set has the Gabrielov property.

\smallskip

However, not only semialgebraic sets, but a very large class of sets has the Gabrielov property. These sets are called {\it tame sets}. The precise definition of these sets and the fact that they satisfy the Gabrielov property can be found, in particular, in \cite{YC}.


\begin{thebibliography}{GGM}

\bibitem {B} Bernstein, D. N. {\em The number of roots of a system of equations}. (Russian)
Funkcional. Anal. i Prilozen. 9 (1975), no. 3, 1-4.

\bibitem {FY} Friedland, O.; Yomdin, Y. {\em An observation on the Tur\'an-Nazarov inequality}. Submitted.

\bibitem {I} Ivanov, L. D. {\em Variatsii mnozhestv i funktsii}. (Russian) [Variations of sets and functions] Edited by A. G. Vitushkin. Izdat. ``Nauka'', Moscow, 1975. 352 pp.

\bibitem {K} Khovanskii, A. G. {\em Fewnomials}. Translated from the Russian by Smilka Zdravkovska. Translations of Mathematical Monographs, 88. American Mathematical Society, Providence, RI, 1991. viii+139 pp.

\bibitem {Ku} Ku\v snirenko, A. G. {\em Newton polyhedra and Bezout's theorem}. (Russian)
Funkcional. Anal. i Prilo\v zen. 10 (1976), no. 3, 82Ð83.

\bibitem {N} Nazarov, F. L. {\em Local estimates for exponential polynomials and their applications to inequalities of the uncertainty principle type}. Algebra i Analiz 5 (1993), no. 4, 3-66; translation in St. Petersburg Math. J. 5 (1994), no. 4, 663-717.

\bibitem {V1} Vitushkin, A. G. {\em O mnogomernyh Variaziyah}. Gostehisdat, Moskow, (1955).

\bibitem {V2} Vitushkin, A. G. {\em Ozenka sloznosti zadachi tabulirovaniya}. Fizmatgiz, Moskow, 1959.

\bibitem {Y} Yomdin, Y. {\em Discrete Remez Inequality}. to appear, Israel J. of Math.

\bibitem {YC} Yomdin, Y.; Comte, G. {\em Tame geometry with application in smooth analysis}. Lecture Notes in Mathematics, 1834. Springer-Verlag, Berlin, 2004. viii+186 pp.

\end{thebibliography}
\end{document}